\documentclass[12pt]{amsart}
\usepackage{amssymb,amsmath,amsthm}
\usepackage{tikz}
\usepackage{cancel}
\usepackage{mathtools,float}
\usepackage[T1]{fontenc}
\usepackage{enumerate}

\usetikzlibrary{arrows,shapes,automata,backgrounds,decorations,petri,
                positioning,patterns}
\colorlet{lightgray}{black!15}

\theoremstyle{plain}
\newtheorem{thm}{Theorem}[section]
\newtheorem{lem}[thm]{Lemma}

\newtheorem{cor}[thm]{Corollary}

\newtheorem{prop}[thm]{Proposition}
\theoremstyle{definition}
\newtheorem{defn}[thm]{Definition}

\newtheorem{ex}[thm]{Example}

\newcommand{\mf}[1]{\mbox{$\mathfrak #1$}}

\newcommand{\av}{\textup{Av}}
\newcommand{\cont}{\textup{Cont}}
\newcommand{\enc}{\textup{\textsf{enc}}}
\newcommand{\boks}[2]{(#1, #2)}   

\title{Coincidence among families of mesh patterns}

\author[Anders Claesson]{Anders Claesson$^{\star}$}
\address{Department of Computer and Information Sciences, University of Strathclyde, Glasgow, UK}
\email{anders.claesson@strath.ac.uk}
\author[Bridget Eileen Tenner]{Bridget Eileen Tenner$^{\dagger}$}
\address{Department of Mathematical Sciences, DePaul University, Chicago, IL, United States}
\email{bridget@math.depaul.edu}
\author[Henning Ulfarsson]{Henning Ulfarsson$^{\star}$}
\address{School of Computer Science, Reykjavik University, Reykjavik, Iceland}
\email{henningu@ru.is}

\thanks{$^{\dagger}$ Research partially supported by a Simons Foundation
  Collaboration Grant for Mathematicians.  $^{\star}$ Research partially
  supported by grant 141761-051 from the Icelandic Research Fund.}


\begin{document}

\begin{abstract}
  Two mesh patterns are coincident if they are avoided by the same set of
  permutations. In this paper, we provide necessary conditions for this
  coincidence, which include having the same set of enclosed diagonals. This
  condition is sufficient to prove coincidence of vincular patterns, although it
  is not enough to guarantee coincidence of bivincular patterns. In addition, we
  provide a generalization of the Shading Lemma (Hilmarsson et al.), a result
  that examined when a square could be added to the mesh of a pattern.

\end{abstract}

\maketitle
\thispagestyle{empty}

\section{Introduction}

The study of patterns in permutations has grown into a rich subfield of
combinatorics.  The original notion of classical patterns --- a
subsequence of symbols having a particular relative order --- has been expanded
to include variations like vincular patterns
\cite{babson-steingrimsson, steingrimsson}, bivincular patterns
\cite{bousquet-melou claesson dukes kitaev}, Bruhat-restricted patterns
\cite{woo yong}, and barred patterns \cite{west-thesis}.  Each of these
flavors adds its own nuance to the ``relative order'' requirement of
classical patterns.  Mesh patterns, which encompass nearly all of the
previously defined flavors were introduced by Br\"and\'en and the first
author in \cite{branden-claesson}.

The main questions about patterns concern how and when a pattern might
be contained (or avoided) by arbitrary permutations.  Two patterns $\pi$
and $\sigma$, perhaps having different flavors, are \emph{coincident} if
the set of permutations that avoid $\pi$ is equal to the set of
permutations that avoid $\sigma$.  Note that this is stronger than
Wilf-equivalence, which requires that the permutations of size
$n$ that avoid $\pi$ is equinumerous to the permutations of
size $n$ that avoid $\sigma$, for all $n$.

It is particularly interesting, and often surprising, when patterns of
different types are coincident.  The second author analyzed this
phenomenon for vincular and barred patterns in
\cite{tenner-coincidental}. In the current paper, we are concerned with
the more general category of mesh patterns, and we analyze how much of a
mesh is required to distinguish it from others.  Put another way, we
examine which portions of a mesh might be superfluous.  The Shading
Lemma of the third author together with Hilmarsson, J\'onsd\'ottir,
Sigur\dh ard\'ottir, and Vi\dh arsd\'ottir~\cite{wilfshort} was the
first general result in this direction. It gives sufficient conditions
for when a square can be added to a mesh. Here we will give a more
powerful version of that lemma, which captures almost all coincidences
of small mesh patterns.

Recent work of the second author described when a mesh pattern is
coincident to a classical pattern \cite{tenner-mesh-classical}.  This
phenomenon is characterized by avoidance of a single configuration in
the mesh, called an \emph{enclosed diagonal}.  In the present work, we
show that enclosed diagonals do, in fact, provide a necessary condition
for two mesh patterns to be coincident, but this property is not
sufficient in general.  On the other hand, we do prove that for some
families $\mathcal{F}$ of mesh patterns, elements of $\mathcal{F}$
are coincident exactly when their enclosed diagonals coincide.

This paper is organized as follows. In Section~\ref{sec:mesh-defn}, we
recall the definition of a mesh pattern. In
Section~\ref{sec:previous-results}, previous results are reviewed. In
Section~\ref{sec:necessary}, we give necessary conditions for mesh
pattern coincidence. In Section~\ref{sec:enc diag enough}, we show that
common enclosed diagonals are sufficient to determine coincidence for
vincular patterns. Examples of pattern families for which coincidence is
not governed by enclosed diagonals are given in Section~\ref{sec:enc
  diag not enough}, including bivincular
patterns. In Section~\ref{sec:ssl}, we establish the Simultaneous Shading
Lemma, and future directions of research are discussed in
Section~\ref{sec:future}.

\section{Mesh patterns}\label{sec:mesh-defn}

The set $\mf{S}_n$ consists of all bijections from $[1,n] =
\{1,\ldots,n\}$ to itself. Each of these bijections is a
\emph{permutation}, and we can denote $w \in \mf{S}_n$ by the word $w =
w(1)w(2)\cdots w(n)$. Two sequences $a_1a_2\cdots a_k$ and $b_1b_2\cdots
b_k$ of real numbers are said to be \emph{order
  isomorphic} if $a_i < a_j$ if and only if $b_i < b_j$ for all $i,j\in
[1,k]$.

\begin{defn}\label{defn:classical}
  Fix $p \in \mf{S}_k$. The permutation
  $w \in \mf{S}_n$ \emph{contains} a $p$-pattern if there exist indices
  $1 \le i_1 < i_2 < \cdots < i_k \le n$ such that
  $
  w(i_1) w(i_2) \cdots w(i_k)
  $
  is order isomorphic to $p$. The subsequence $w(i_1)w(i_2) \cdots
  w(i_k)$ is an \emph{occurrence} of $p$ in $w$. If $w$ does not contain
  $p$, then $w$ \emph{avoids} $p$.
\end{defn}

Definition~\ref{defn:classical} is the classical framework for
permutation patterns.  The following example demonstrates this classical
containment and avoidance.

\begin{ex}\label{ex:classical patterns}
  The permutation $42135$ contains five occurrences of the pattern
  $213$, namely $425$, $415$, $435$, $213$, and $215$.
  The permutation $42135$ avoids the pattern $132$.
\end{ex}

A permutation $w\in\mf{S}_n$ can also be represented graphically by
plotting the points
\[
G(w) = \bigl\{(i,w(i)) : i \in [1,n]\bigr\}.
\]
Note that $G(w)$ is a subset of the Cartesian product $[1,n]^2 =
[1,n]\times [1,n]$.  Classical pattern containment and avoidance can be
formulated in terms of these graphs.  Namely, $w$ contains a $p$-pattern
if and only if $G(w)$ contains a copy of $G(p)$, as shown in the
following examples.

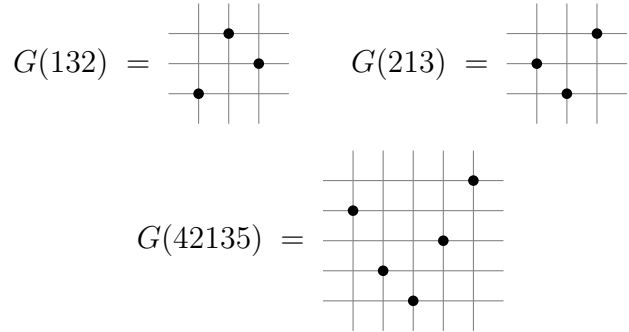
\begin{figure}[htbp]
  \begin{gather*}
  G(132) \;=\;
  \begin{tikzpicture}[scale=.4, baseline={([yshift=-3pt]current bounding box.center)}]
    \foreach \x in {1,2,3} {
      \draw[gray] (0,\x) -- (4,\x);
      \draw[gray] (\x,0) -- (\x,4);
    }
    \foreach \x in {(1,1),(2,3),(3,2)} {\fill[black] \x circle (5pt);}
  \end{tikzpicture}\qquad
  G(213) \;=\;
  \begin{tikzpicture}[scale=.4, baseline={([yshift=-3pt]current bounding box.center)}]
    \foreach \x in {1,2,3} {
      \draw[gray] (0,\x) -- (4,\x);
      \draw[gray] (\x,0) -- (\x,4);
    }
    \foreach \x in {(1,2),(2,1),(3,3)} {\fill[black] \x circle (5pt);}
  \end{tikzpicture}\\[1ex]
  G(42135) \;=\;
  \begin{tikzpicture}[scale=.4, baseline={([yshift=-3pt]current bounding box.center)}]
    \foreach \x in {1,2,3,4,5} {
      \draw[gray] (0,\x) -- (6,\x);
      \draw[gray] (\x,0) -- (\x,6);
    }
    \foreach \x in {(1,4),(2,2),(3,1),(4,3),(5,5)} {\fill[black] \x circle (5pt);}
  \end{tikzpicture}
\end{gather*}
\caption{The permutations $132$, $213$, $42135$ and their graphs.}\label{fig:graph of 42135}
\end{figure}

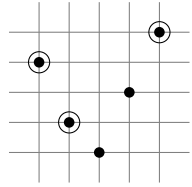
\begin{figure}[htbp]
  \[
  \begin{tikzpicture}[scale=.4, baseline={([yshift=-3pt]current bounding box.center)}]
    \foreach \x in {1,2,3,4,5} {
      \draw[gray] (0,\x) -- (6,\x);
      \draw[gray] (\x,0) -- (\x,6);
    }
    \foreach \x in {(1,4),(2,2),(3,1),(4,3),(5,5)} {\fill[black] \x circle (5pt);}
    \foreach \x in {(1,4),(2,2),(5,5)} {\draw \x circle (10pt);}
  \end{tikzpicture}
  \]
  \caption{A particular occurrence of $213$ in $42135$.}\label{fig:213-in-42135}
\end{figure}

\begin{ex}
  In Figure~\ref{fig:graph of 42135}, we show the graphs $G(132)$,
  $G(213)$, and $G(42135)$. The graph $G(42135)$ contains $G(213)$ in
  five ways.  These correspond to the five occurrences of $213$ in
  $42135$, as discussed in Example~\ref{ex:classical patterns}. The copy
  of $G(213)$ that corresponds to the occurrence $425$ is marked in
  Figure~\ref{fig:213-in-42135}. The graph $G(42135)$ does not
  contain $G(132)$, which corresponds to the fact that $42135$ avoids
  the pattern $132$.
\end{ex}

\begin{defn}\label{defn:mesh-pattern}
  A \emph{mesh pattern} is an ordered pair $(p,R)$ where $p \in
  \mf{S}_k$ is a permutation, and $R$ is a subset of the $(k+1)^2$ unit
  squares in $[0,k+1]^2$.  The set $R$ is the \emph{mesh}, and elements
  (squares) of the mesh are indexed by their lower-left corners; that
  is, $(a,b) \in R$ refers to the square $[a,a+1]\times[b,b+1]$.  The
  mesh pattern $(p,R)$ is depicted graphically by drawing $G(p)$ and
  shading all squares of the mesh $R$.
\end{defn}

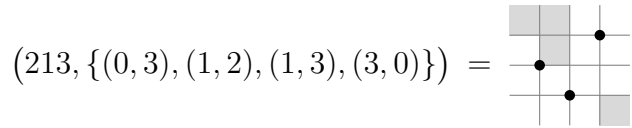
\begin{figure}[htbp]
  \[
  \big(213,\{(0,3),(1,2),(1,3),(3,0)\}\big) \;=\;
  \begin{tikzpicture}[scale=.4, baseline={([yshift=-3pt]current bounding box.center)}]
    \foreach \x in {(0,3),(1,2),(1,3),(3,0)} {\fill[lightgray] \x rectangle ++ (1,1);}
    \foreach \x in {1,2,3} {
      \draw[gray] (0,\x) -- (4,\x);
      \draw[gray] (\x,0) -- (\x,4);
    }
    \foreach \x in {(1,2),(2,1),(3,3)} {\fill[black] \x circle (5pt);}
  \end{tikzpicture}
  \]
  \caption{A mesh pattern.}\label{fig:a-mesh-pattern}
\end{figure}

To illustrate Definition~\ref{defn:mesh-pattern}, we give a
particular mesh pattern in Figure~\ref{fig:a-mesh-pattern}.

For a permutation $w$ to contain a mesh pattern $(p,R)$, this $w$ must contain
an occurrence of the classical pattern $p$ that does not ``interfere'' with the
mesh.  To clarify this interference requirement, we must first explain how the
mesh of $(p,R)$ appears in the graph $G(w)$ relative to a particular occurrence
of $p$.

\begin{defn}\label{defn:corresponding regions}
  Let $w \in \mf{S}_n$ and $p \in \mf{S}_k$.  Suppose that $w$ contains
  a $p$-pattern in positions $i_1 < i_2 < \cdots < i_k$.  By convention,
  we additionally set $i_0 = p(i_0) = 0$ and $i_{k+1} = p(i_{k+1}) =
  n+1$. Let $a,b\in [0,k]$. The square $[a,a+1]\times[b,b+1]$ in $G(p)$
  \emph{corresponds} to the following rectangle in $G(w)$:
  \[
    \Big[\,i_a,\,i_{a+1}\,\Big]\times
    \Big[\,w\big(i_{p^{-1}(b)}\big),\,w\big(i_{p^{-1}(b+1)}\big)\,\Big].
    \]
  \end{defn}

\begin{ex}
  The permutation $w = 42135$ contains a $213$-pattern in positions
  $\{1,3,5\}$. The square $[1,2]\times[3,4] \subset [0,4]^2$ in $G(213)$
  corresponds to the rectangle $[1,4]\times[3,5] \subset [0,6]^2$ in
  $G(w)$.  In Figure~\ref{fig:a-rectangle} this particular $213$-pattern in
  $w$ is marked by circles around the points, and the rectangle
  $[1,4]\times[3,5]$ has been outlined and shaded.

  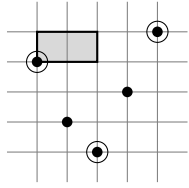
\begin{figure}[htbp]
    \[
    \begin{tikzpicture}[scale=.4]
    \fill[lightgray] (1,4) rectangle (3,5);
    \foreach \x in {1,2,3,4,5} {
      \draw[gray] (0,\x) -- (6,\x);
      \draw[gray] (\x,0) -- (\x,6);
    }
    \foreach \x in {(1,4),(2,2),(3,1),(4,3),(5,5)} {\fill[black] \x circle (5pt);}
    \foreach \x in {(1,4),(3,1),(5,5)}{\draw \x circle (10pt);}
    \draw[thick] (1,4) rectangle (3,5);
  \end{tikzpicture}
  \]
  \caption{The permutation $42135$ with a marked $213$-pattern and the
    rectangle $[1,4]\times[3,5]$ outlined and
    shaded.}\label{fig:a-rectangle}
\end{figure}
\end{ex}

\begin{defn}\label{defn:mesh}
  A permutation $w$ \emph{contains} a mesh pattern $(p,R)$ if $w$
  contains an occurrence of $p$ for which the subregion of $[0,n+1]^2$
  corresponding to the mesh $R$ contains no points of $G(w)$.  If there
  is no such occurrence of $p$ in $w$, then $w$ \emph{avoids} the mesh
  pattern $(p,R)$.
\end{defn}

It is important to observe that the subregion mentioned in
Definition~\ref{defn:mesh} depends on the particular occurrence of $p$
in $w$.

\begin{ex}
  The permutation $42135$ contains five different $213$-patterns. On the other
  hand, it contains the mesh pattern $(213,\{(0,3),(1,2),(1,3),(3,0)\})$ in only
  four ways.  These are depicted in Figure~\ref{fig:213-patterns} where thick
  lines indicate how the four squares of the mesh appear in $G(42135)$.

  \begin{figure}[htbp]
    \begin{equation*}
      \begin{tikzpicture}[scale=0.38]
        \foreach \x in {(0,5),(1,4),(1,5),(5,0),(5,1)}{\fill[lightgray] \x rectangle ++(1,1);}
        \foreach \x in {1,2,3,4,5} {
          \draw[gray] (0,\x) -- (6,\x);
          \draw[gray] (\x,0) -- (\x,6);
        }
        \foreach \x in {(1,4),(2,2),(3,1),(4,3),(5,5)} {\fill[black] \x circle (5pt);}
        \foreach \x in {(1,4),(2,2),(5,5)}{\draw \x circle (10pt);}
        \draw[thick] (0,6) -- ++(2,0) -- ++(0,-2) -- ++(-1,0) -- ++(0,1) -- ++(-1,0) -- cycle;
        \draw[thick] (1,5) -- +(0,1) (1,5) -- +(1,0);
        \draw[thick] (5,0) -- ++(0,2) -- ++(1,0) -- ++(0,-2) -- cycle;
      \end{tikzpicture}
      \qquad\;
      \begin{tikzpicture}[scale=0.38]
        \fill[lightgray] (1,4) rectangle (3,5);
        \fill[lightgray] (0,5) rectangle (3,6);
        \fill[lightgray] (5,0) rectangle (6,1);
        \foreach \x in {1,2,3,4,5} {
          \draw[gray] (0,\x) -- (6,\x);
          \draw[gray] (\x,0) -- (\x,6);
        }
        \foreach \x in {(1,4),(2,2),(3,1),(4,3),(5,5)} {\fill[black] \x circle (5pt);}
        \foreach \x in {(1,4),(3,1),(5,5)}{\draw \x circle (10pt);}
        \draw[thick] (0,5) -- (0,6) -- (3,6) -- (3,4) -- (1,4) -- (1,5) -- cycle;
        \draw[thick] (1,5) -- (3,5) (1,5) -- (1,6);
        \draw[thick] (5,0) rectangle (6,1);
      \end{tikzpicture}
      \qquad\;
      \begin{tikzpicture}[scale=0.38]
        \fill[lightgray] (1,4) rectangle (4,5);
        \fill[lightgray] (0,5) rectangle (4,6);
        \fill[lightgray] (5,0) rectangle (6,3);
        \foreach \x in {1,2,3,4,5} {
          \draw[gray] (0,\x) -- (6,\x);
          \draw[gray] (\x,0) -- (\x,6);
        }
        \foreach \x in {(1,4),(2,2),(3,1),(4,3),(5,5)} {\fill[black] \x circle (5pt);}
        \foreach \x in {(1,4),(4,3),(5,5)}{\draw \x circle (10pt);}
        \draw[thick] (0,5) -- (0,6) -- (4,6) -- (4,4) -- (1,4) -- (1,5) -- cycle;
        \draw[thick] (1,5) -- (1,6) (1,5) -- (4,5);
        \draw[thick] (5,0) rectangle (6,3);
      \end{tikzpicture}
      \qquad\;
      \begin{tikzpicture}[scale=0.38]
        \fill[lightgray] (0,5) rectangle (2,6);
        \fill[lightgray] (2,2) rectangle (3,6);
        \fill[lightgray] (5,0) rectangle (6,1);
        \foreach \x in {1,2,3,4,5} {
          \draw[gray] (0,\x) -- (6,\x);
          \draw[gray] (\x,0) -- (\x,6);
        }
        \foreach \x in {(1,4),(2,2),(3,1),(4,3),(5,5)} {\fill[black] \x circle (5pt);}
        \foreach \x in {(2,2), (3,1), (5,5)}{\draw \x circle (10pt);}
        \draw[thick] (0,5) rectangle (3,6);
        \draw[thick] (2,2) rectangle (3,6);
        \draw[thick] (5,0) rectangle (6,1);
      \end{tikzpicture}
    \end{equation*}
    \caption{The 4 occurrences, in the permutation $42135$, of the mesh
      pattern
      $(213,\{(0,3),(1,2),(1,3),(3,0)\})$.}\label{fig:213-patterns}
  \end{figure}
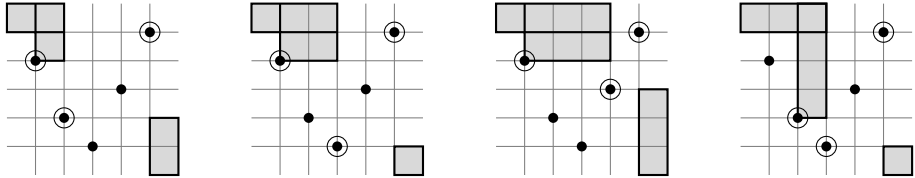
  The remaining occurrence of $213$ in $42135$ does not obey the
  restrictions of the mesh because the shaded region includes the point
  $(1,4) \in G(42135)$ as shown in Figure~\ref{fig:not-213-patterns}.

  \begin{figure}[htbp]
    $$
    \begin{tikzpicture}[scale=0.38]
      \fill[lightgray] (0,3) rectangle (2,6);
      \fill[lightgray] (2,2) rectangle (3,6);
      \fill[lightgray] (4,0) rectangle (6,1);
      \foreach \x in {1,2,3,4,5} {
        \draw[gray] (0,\x) -- (6,\x);
        \draw[gray] (\x,0) -- (\x,6);
      }
      \foreach \x in {(1,4),(2,2),(3,1),(4,3),(5,5)} {
        \fill[black] \x circle (5pt);
      }
      \foreach \x in {(2,2), (3,1), (4,3)}{
        \draw \x circle (10pt);
      }
      \draw[thick] (0,3) rectangle (3,6);
      \draw[thick] (2,2) rectangle (3,6);
      \draw[thick] (4,0) rectangle (6,1);
    \end{tikzpicture}
    $$
    \caption{A $213$-pattern in the permutation $42135$ that is not an
      occurrence of the mesh pattern
      $(213,\{(0,3),(1,2),(1,3),(3,0)\})$.
    }\label{fig:not-213-patterns}
  \end{figure}
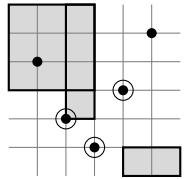
\end{ex}

\begin{defn}\label{defn:av and cont}
  For a pattern $\pi$ of any flavor, let $\av(\pi)$ be the set of
  permutations that avoid $\pi$ and let $\cont(\pi)$ be the set of
  permutations that contain $\pi$.  Similarly, if $\Pi$ is a collection
  of patterns, then $\av(\Pi) = \bigcap_{\pi \in \Pi} \av(\pi)$ and
  $\cont(\Pi) = \bigcup_{\pi \in \Pi} \cont(\pi) =
  \big(\bigcup_n\mf{S}_n\big) \setminus \av(\Pi)$.
\end{defn}

The purpose of this article is to examine when the sets of
Definition~\ref{defn:av and cont} coincide for two different mesh
patterns.  This is captured by the following definition.

\begin{defn}
  Two patterns $\pi$ and $\sigma$ are \emph{coincident} if $\av(\pi) =
  \av(\sigma)$, or, equivalently, if $\cont(\pi) = \cont(\sigma)$. This
  coincidence will be denoted $\pi \asymp \sigma$.
\end{defn}

\section{Mesh patterns can be coincident --- examples and previous results}\label{sec:previous-results}

Mesh pattern coincidence is not a rare phenomenon, as demonstrated in Figure~\ref{fig:mesh coincidence with classical}.

\begin{figure}[htbp]
  $$
  \begin{tikzpicture}[scale=.4, baseline={([yshift=-3pt]current bounding box.center)}]
    \foreach \x in {(1,2),(2,0)} {
      \fill[lightgray] \x rectangle ++ (1,1);
    }
    \foreach \x in {1,2,3} {
      \draw[gray] (0,\x) -- (4,\x);
      \draw[gray] (\x,0) -- (\x,4);
    }
    \foreach \x in {(1,2),(2,1),(3,3)} {
      \fill[black] \x circle (5pt);
    }
  \end{tikzpicture}
  \;\asymp\;
  \begin{tikzpicture}[scale=.4, baseline={([yshift=-3pt]current bounding box.center)}]
    \foreach \x in {(1,2)} {\fill[lightgray] \x rectangle ++ (1,1);}
    \foreach \x in {1,2,3} {
      \draw[gray] (0,\x) -- (4,\x);
      \draw[gray] (\x,0) -- (\x,4);
    }
    \foreach \x in {(1,2),(2,1),(3,3)} {\fill[black] \x circle (5pt);}
  \end{tikzpicture}
  \;\asymp\;
  \begin{tikzpicture}[scale=.4, baseline={([yshift=-3pt]current bounding box.center)}]
    \foreach \x in {(2,0)} {\fill[lightgray] \x rectangle ++ (1,1);}
    \foreach \x in {1,2,3} {
      \draw[gray] (0,\x) -- (4,\x);
      \draw[gray] (\x,0) -- (\x,4);
    }
    \foreach \x in {(1,2),(2,1),(3,3)} {\fill[black] \x circle (5pt);}
  \end{tikzpicture}
  \;\asymp\;
  \begin{tikzpicture}[scale=.4, baseline={([yshift=-3pt]current bounding box.center)}]
    \foreach \x in {1,2,3} {
      \draw[gray] (0,\x) -- (4,\x);
      \draw[gray] (\x,0) -- (\x,4);
    }
    \foreach \x in {(1,2),(2,1),(3,3)} {\fill[black] \x circle (5pt);}
  \end{tikzpicture}
  $$
  \caption{Mesh patterns coincident with a classical pattern.}
  \label{fig:mesh coincidence with classical}
\end{figure}
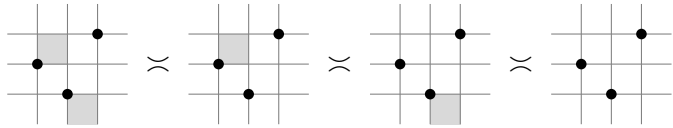

The property underlying this example is the main result of
\cite{tenner-mesh-classical}, which we will review in this section. This
requires the notion of an enclosed diagonal.

\begin{defn}\label{defn:enclosed}
  Let $(p,R)$ be a mesh pattern and let $a$, $b$, and $c$ be nonnegative
  integers. Then $D = \big\{(a+i,b+i) : i \in [0,c]\big\}\subseteq R$
  is an \emph{enclosed NE-diagonal} if
  \begin{align*}
    \big\{(a+i,b+i) : i\in [0,c+1]\big\} \cap G(p) &=
    \big\{(a+i,b+i) : i\in [1,c]\big\} \\
    \intertext{Similarly, $D = \big\{(a+i,b-i) : i \in [0,c]\big\}\subseteq R$
      is an \emph{enclosed SE-diagonal} if}
    \big\{(a+i,b+1-i) : i\in [0,c+1]\big\} \cap G(p) &=
    \big\{(a+i,b+1-i) : i\in [1,c]\big\}.
  \end{align*}
  If $D$ has either of these types, then $D$ can be called, simply, an
  \emph{enclosed diagonal}. The \emph{length} of an enclosed diagonal is
  $c+1$, and enclosed diagonals of length greater than one are
  \emph{proper}.  An enclosed diagonal of length $1$ is both an enclosed
  NE-diagonal and an enclosed SE-diagonal, and such a square will be
  called \emph{pointless} because it intersects no elements of the graph $G(p)$.
\end{defn}

Definition~\ref{defn:enclosed} is illustrated in Figure~\ref{fig:enclosed}. We
note that our definition of enclosed notation differs from that of
\cite{tenner-mesh-classical} because in the present setting, it is more natural
to associate an enclosed diagonal with the shaded squares of the mesh.

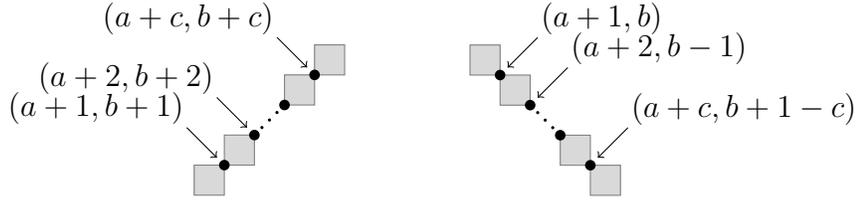
\begin{figure}[H]
  $$
  \begin{tikzpicture}[scale=.4]
    \foreach \x in {0,1,3,4} {
      \fill[lightgray] (\x,\x) rectangle ++(1,1);
      \draw[gray] (\x,\x) rectangle ++(1,1);
    } \foreach \x in {1,2,3,4} {
      \fill[black] (\x,\x) circle (5pt);
    }
    \foreach \x in {2.25,2.5,2.75} {
      \fill[black] (\x,\x) circle (1.5pt);
    }
    \draw (0,2) node[above left] {$(a+1,b+1)$};
    \draw[->] (-.25,2.25) -- (.75,1.25);
    \draw (1,3) node[above left] {$(a+2,b+2)$};
    \draw[->] (.75,3.25) -- (1.75,2.25);
    \draw (3,5) node[above left] {$(a+c,b+c)$};
    \draw[->] (2.75,5.25) -- (3.75,4.25);
  \end{tikzpicture}
  \qquad\qquad
  \begin{tikzpicture}[scale=.4]
    \foreach \x in {0,1,3,4} {
      \fill[lightgray] (\x,-\x) rectangle ++(1,-1);
      \draw[gray] (\x,-\x) rectangle ++(1,-1);}
    \foreach \x in {1,2,3,4} {\fill[black] (\x,-\x) circle (5pt);}
    \foreach \x in {2.25,2.5,2.75} {\fill[black] (\x,-\x) circle (1.5pt);}
    \draw (5,-3) node[above right] {$(a+c,b+1-c)$}; \draw[->] (5.25,-2.75) -- (4.25,-3.75);
    \draw (3,-1) node[above right] {$(a+2,b-1)$}; \draw[->] (3.25,-.75) -- (2.25,-1.75);
    \draw (2,0) node[above right] {$(a+1,b)$}; \draw[->] (2.25,.25) -- (1.25,-.75);
  \end{tikzpicture}
  $$
  \caption{Proper enclosed diagonals in a mesh pattern $(p,R)$.  A point
    is in $G(p)$ if and only if it is marked
    $\bullet$.}\label{fig:enclosed}
\end{figure}

\begin{ex}\label{ex:enclosed diagonal}
  The mesh pattern $\big(231,\{(1,1),(2,0),(3,1)\}\big)$ depicted below
  $$
  
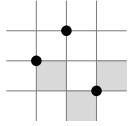
\begin{tikzpicture}[scale=.4, baseline={([yshift=-2.5pt]current bounding box.center)}]
    \foreach \x in {(1,1),(2,0),(3,1)} {
      \fill[lightgray] \x rectangle ++(1,1);
    }
    \foreach \x in {1,2,3} {
      \draw[gray] (\x,0) -- (\x,4);
      \draw[gray] (0,\x) -- (4,\x);
    }
    \foreach \x in {(1,2),(2,3),(3,1)} {
      \fill[black] \x circle (5pt);
    }
  \end{tikzpicture}
  $$
  has one enclosed diagonal, namely $\{(2,0), (3,1)\} \subset R$. This is an
  enclosed NE-diagonal and has length $2$.
\end{ex}

The next lemma follows immediately from the definition of enclosed
diagonals and shows that these configurations in a mesh pattern are, in
a sense, stable.

\begin{lem}\label{lem:enclosed diagonals are stable}
  Let $(p,R)$ be a mesh pattern.
  \begin{enumerate}[(a)]
  \item If $R' \subseteq R$, then every enclosed diagonal in the mesh
    pattern $(p,R')$ is an enclosed diagonal in $(p,R)$.
  \item Each square in $R$ belongs to at most one enclosed diagonal.
  \end{enumerate}
\end{lem}

It will be useful to look at the collection of enclosed diagonals in a
given mesh pattern $(p,R)$, and we will denote this set
$$\enc{(p,R)}.$$

Previous work by the second author characterized when an entire mesh is
superfluous; that is, when a mesh pattern is equivalent to its
underlying classical pattern. The coincidences in
Figure~\ref{fig:mesh coincidence with classical}, for example, are explained by this
result.

\begin{thm}[{\cite[Theorem 3.5]{tenner-mesh-classical}}]\label{thm:classical}
  A mesh pattern $(p,R)$ is coincident to a classical pattern if and
  only if $\enc{(p,R)} = \emptyset$.  If this is the case, then $(p,R)
  \asymp p$.
\end{thm}

Theorem~\ref{thm:classical} suggests that enclosed diagonals are the
crux to understanding coincidence between mesh patterns, and that
perhaps mesh patterns are coincident if and only if they have the same
enclosed diagonals.  Unfortunately, as the following example shows, this
is not the case.

\begin{ex} \label{ex:surprising non-coincidence}
  The coincidences
  $$
  
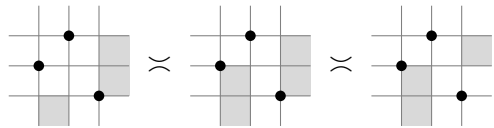
\begin{tikzpicture}[scale=.4, baseline={([yshift=-2.5pt]current bounding box.center)}]
    \foreach \x in {(1,0),(3,1),(3,2)} {
      \fill[lightgray] \x rectangle ++ (1,1);
    }
    \foreach \x in {1,2,3} {
      \draw[gray] (0,\x) -- (4,\x);
      \draw[gray] (\x,0) -- (\x,4);
    }
    \foreach \x in {(1,2),(2,3),(3,1)} {
      \fill[black] \x circle (5pt);
    }
  \end{tikzpicture}
  \;\asymp\;
  \begin{tikzpicture}[scale=.4, baseline={([yshift=-2.5pt]current bounding box.center)}]
    \foreach \x in {(1,0),(1,1),(3,1),(3,2)} {\fill[lightgray] \x rectangle ++ (1,1);}
    \foreach \x in {1,2,3} {
      \draw[gray] (0,\x) -- (4,\x);
      \draw[gray] (\x,0) -- (\x,4);
    }
    \foreach \x in {(1,2),(2,3),(3,1)} {\fill[black] \x circle (5pt);}
  \end{tikzpicture}
  \;\asymp\;
  \begin{tikzpicture}[scale=.4, baseline={([yshift=-2.5pt]current bounding box.center)}]
    \foreach \x in {(1,0),(1,1),(3,2)} {\fill[lightgray] \x rectangle ++ (1,1);}
    \foreach \x in {1,2,3} {
      \draw[gray] (0,\x) -- (4,\x);
      \draw[gray] (\x,0) -- (\x,4);
    }
    \foreach \x in {(1,2),(2,3),(3,1)} {\fill[black] \x circle (5pt);}
  \end{tikzpicture}
  $$
  follow from the Shading Lemma (\cite[Lemma 3.11]{wilfshort}),
  reviewed below.  Note, however, that these three mesh patterns in are
  not coincident to
  $$
  
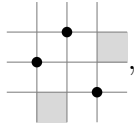
\begin{tikzpicture}[scale=.4, baseline={([yshift=-2.5pt]current bounding box.center)}]
    \foreach \x in {(1,0),(3,2)} {
      \fill[lightgray] \x rectangle ++ (1,1);
    }
    \foreach \x in {1,2,3} {
      \draw[gray] (0,\x) -- (4,\x);
      \draw[gray] (\x,0) -- (\x,4);
    }
    \foreach \x in {(1,2),(2,3),(3,1)} {
      \fill[black] \x circle (5pt);
    }
  \end{tikzpicture},
  $$
  the mesh pattern containing only the enclosed diagonals. This is because $42513$ avoids the first three patterns but contains
  this last pattern, even though the enclosed diagonals are $\{(1,0)\}$
  and $\{(3,2)\}$ in each case.
\end{ex}

The third author and Hilmarsson, J\'onsd\'ottir, Sigur\dh ard\'ottir,
and Vi\dh arsd\'ottir have given sufficient conditions for when $(p,R)
\asymp (p, R')$ where $R$ and $R'$ differ by only one square.

\begin{lem}[{Shading Lemma (northeast) \cite[Lemma 3.11]{wilfshort}}]\label{lem:shading}
  Let $(p,R)$ be a mesh pattern. Suppose that the following conditions all hold.
  \begin{itemize}
  \item Neither the square $\boks{i}{p(i)}$ nor the square $\boks{i-1}{p(i)-1}$
  is in $R$.
  \item At most one of the squares $\boks{i}{p(i)-1}$ and $\boks{i-1}{p(i)}$
  is in $R$.
  \item For all $x \not\in \{i-1,i\}$, if $\boks{x}{p(i)-1} \in R$ then
  $\boks{x}{p(i)} \in R$.
  \item For all $y \not\in \{p(i)-1, p(i)\}$, if $\boks{i-1}{y} \in R$ then
  $\boks{i}{y} \in R$.
  \end{itemize}
  Then $(p,R) \asymp (p,R\cup\{\boks{i}{p(i)} \})$.
\end{lem}

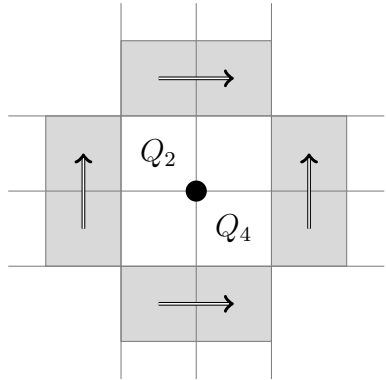
\begin{figure}[htbp]
  $$
  \begin{tikzpicture}[scale=1]
    \filldraw[fill=lightgray,draw=gray] (1,0) rectangle (3,1);
    \filldraw[fill=lightgray,draw=gray] (0,1) rectangle (1,3);
    \filldraw[fill=lightgray,draw=gray] (1,3) rectangle (3,4);
    \filldraw[fill=lightgray,draw=gray] (3,1) rectangle (4,3);
    \foreach \x in {1,2,3} {
      \draw[gray] (\x,-0.5) -- (\x,4.5);
      \draw[gray] (-0.5,\x) -- (4.5,\x);
    }
    \draw[->, double] (1.5,0.5) -- (2.5,0.5);
    \draw[->, double] (0.5,1.5) -- (0.5,2.5);
    \draw[->, double] (1.5,3.5) -- (2.5,3.5);
    \draw[->, double] (3.5,1.5) -- (3.5,2.5);
    \fill[black] (2,2) circle (4pt);
    \draw (1.5, 2.5) node {$Q_2$};
    \draw (2.5, 1.5) node {$Q_4$};
  \end{tikzpicture}
  $$
  \caption{These are the conditions required to apply
    Lemma~\ref{lem:shading}, where at most one of $Q_2$ or $Q_4$ is
    shaded.}\label{fig:shadinglemma}
\end{figure}

Note that there are northwest, southeast, and southwest analogues for
Lemma~\ref{lem:shading}, which determine whether the other three
squares incident to an element of $G(p)$ can be added to $R$ in such a
way so as to yield a coincident mesh pattern.

Example~\ref{ex:surprising non-coincidence} yields the following somewhat
surprising result.

\begin{cor}\label{cor:no minimal mesh}
  Fix a mesh pattern $(p,R)$.  The set $\{R' : (p,R) \asymp (p,R')\}$
  need not have a unique smallest element when ordered by set
  containment of the meshes.
\end{cor}

\section{Necessary conditions for mesh pattern coincidence}\label{sec:necessary}

Consider mesh patterns $(p,R)$ and $(p',R')$, and suppose that
$(p,R)\asymp (p',R')$.  Then two conditions must hold -- one condition
about $p$ and $p'$, and another about $R$ and $R'$.

\begin{lem}\label{lem:p=p'}
  If $(p,R)\asymp (p',R')$, then $p = p'$.
\end{lem}

\begin{proof}
  Suppose that $p$ is of length $k$. Let $\pi=(p,R)$ and
  $\pi'=(p',R')$. Clearly, $\mf{S}_k \cap \cont(\pi) = \{p\}$, and
  $\mf{S}_i \cap \cont(\pi) = \emptyset$ for all $i < k$.  In other
  words, there is a unique element of $\cont(\pi)$ using a minimal
  number of symbols, and that element is $p$.  An analogous statement
  holds for $\cont(\pi')$ and $p'$.  Further, if $\pi \asymp \pi'$ then
  $\cont(\pi) = \cont(\pi')$, and thus $p=p'$.
\end{proof}

We now consider coincident mesh patterns $(p,R) \asymp (p,R')$ with the same
underlying classical pattern. Example~\ref{ex:surprising non-coincidence} shows
that having the same enclosed diagonals is not a sufficient condition for coincidence,
but the lemma below shows that it is necessary.

\begin{lem}\label{lem:same enc diag}
  If $(p,R) \asymp (p,R')$, then $\enc{(p,R)} = \enc{(p,R')}$.
\end{lem}

\begin{proof}
  Fix $p \in \mf{S}_k$ and suppose that $\pi=(p,R)$ has an enclosed
  diagonal that does not exist in $(p,R')$.  Without loss of generality,
  let this be a NE-diagonal comprised of
  $D = \{(a+i,b+i) : 0 \le i \le c\} \subseteq R$.
  Then, as a word,
  \[
    p \;=\;
    p(1)\, \cdots\,p(a)\,(b+1)\,(b+2)\,\cdots\,(b+c)\,p(a+c+1)\,\cdots\,p(k),
    \]
  where $p(a) \neq b$ and $p(a+c+1) \neq b+c+1$.

  Define $q \in \mf{S}_{k+1}$ to be the permutation that is order
  isomorphic to the word obtained by inserting $b+1/2$ between $p(a)$
  and $p(a+1)=b+1$ in the word for $p$. By construction, the only
  $p$-patterns in $q$ are obtained by excluding a single letter $q(x)$,
  where $a+1 \le x \le a+c+1$.  Thus there is no way to find an
  occurrence of $p$ that can be drawn in $G(q)$ so that the shaded
  region corresponding to the mesh of $(p,R)$ does not contain $(x,q(x))
  \in G(q)$.  Therefore $q \in \av(\pi)$.

  Now consider the mesh pattern $\pi'=(p,R')$, and take $t$ in $[0,c]$
  so that $(a+t,b+t) \not\in R'$.  Then omitting the letter $q(a+t+1)$
  from $q$ yields a $p$-pattern.  For this occurrence of $p$, the only
  element of $G(q)$ we need to worry about is $(a+t+1,q(a+t+1))$, but
  this point does not lie in a shaded region corresponding to the mesh
  of $\pi'$ because $(a+t,b+t) \not\in R'$.  Thus $q \in \cont(\pi')$,
  and so $\pi \not\asymp \pi'$.
\end{proof}

It is helpful to see an example illustrating the proof of
Lemma~\ref{lem:same enc diag}.

\begin{ex}
  Let $p = 25134\in\mf{S}_5$.  Suppose that $\{(3,2),(4,3),(5,4)\}
  \subset R$ and $(4,3) \not\in R'$. Let $\pi=(p,R)$ and $\pi'=(p,R')$.
  That is,
  $$\pi \;=\;
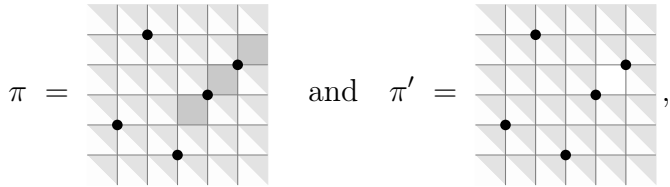
    \begin{tikzpicture}[scale=.4, xscale=-1, baseline={([yshift=-2.5pt]current bounding box.center)}]
      \foreach \x in {0,1,2,3,4,5} {
        \foreach \y in {0,1,2,3,4,5} {
          \filldraw[black!11] (\x,\y) -- +(0,1) -- +(1,1);
          \filldraw[black!1 ] (\x,\y) -- +(1,0) -- +(1,1);
        };
      };
      \foreach \x in {(0,4),(1,3),(2,2)} {
        \fill[white] \x rectangle ++(1,1);
        \fill[black!21] \x rectangle ++(1,1);
      }
      \foreach \x in {1,2,3,4,5} {
        \draw[gray] (0,\x) -- (6,\x);
        \draw[gray] (\x,0) -- (\x,6);
      }
      \foreach \x in {(1,4),(2,3),(3,1),(4,5),(5,2)} {
        \fill[black] \x circle (5pt);
      }
    \end{tikzpicture}
    \;\quad\text{and}\quad
    \pi' \;=\;
    \begin{tikzpicture}[scale=.4, xscale=-1, baseline={([yshift=-2.5pt]current bounding box.center)}]
      \foreach \x in {0,1,2,3,4,5} {
        \foreach \y in {0,1,2,3,4,5} {
          \filldraw[black!11] (\x,\y) -- +(0,1) -- +(1,1);
          \filldraw[black!1 ] (\x,\y) -- +(1,0) -- +(1,1);
        };
      };
      \fill[white] (1,3) rectangle (2,4);
      \foreach \x in {1,2,3,4,5} {
        \draw[gray] (0,\x) -- (6,\x);
        \draw[gray] (\x,0) -- (\x,6);
      }
      \foreach \x in {(1,4),(2,3),(3,1),(4,5),(5,2)} {\fill[black] \x circle (5pt);}
    \end{tikzpicture}\,,
    $$
    where a half-filled square,\!  \tikz[scale=0.27]{
      \filldraw[black!11] (0,0) -- +(0,1) -- +(1,1); \filldraw[black!1 ]
      (0,0) -- +(1,0) -- +(1,1); \draw[gray] (0,0) rectangle (1,1); }\,,
    indicates a square whose shading status --- that is, whose presence in
    the mesh --- is undeclared.  The proof of Lemma~\ref{lem:same enc
      diag} constructs the permutation $q = 261345\in \mf{S}_6$, which has
    three $p$-patterns: $26145$, $26135$, and $26134$. In each of
    these occurrences, there is an element of $G(q)$ located in the
    region corresponding to the mesh from $\pi$. This is shown below.
  $$
  
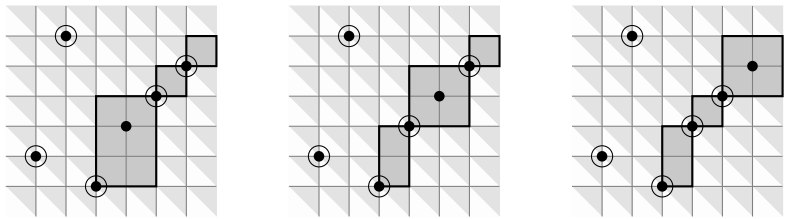
\begin{tikzpicture}[scale=.4, xscale=-1]
    \foreach \x in {0,1,2,3,4,5,6} {
      \foreach \y in {0,1,2,3,4,5,6} {
        \filldraw[black!11] (\x,\y) -- +(0,1) -- +(1,1);
        \filldraw[black!1 ] (\x,\y) -- +(1,0) -- +(1,1);
      };
    };
    \fill[black!21] (0,5) rectangle (1,6);
    \fill[black!21] (1,4) rectangle (2,5);
    \fill[black!21] (2,1) rectangle (4,4);
    \foreach \x in {1,2,3,4,5,6} {
      \draw[gray] (0,\x) -- (7,\x);
      \draw[gray] (\x,0) -- (\x,7);
    }
    \foreach \x in {(1,5),(2,4),(3,3),(4,1),(5,6),(6,2)} {
      \fill[black] \x circle (5pt);
    }
    \foreach \x in {(1,5),(2,4),(4,1),(5,6),(6,2)} {
      \draw \x circle (10pt);
    }
    \draw[thick] (0,5) rectangle (1,6);
    \draw[thick] (1,4) rectangle (2,5);
    \draw[thick] (2,1) rectangle (4,4);
  \end{tikzpicture}
  \qquad\;
  \begin{tikzpicture}[scale=.4, xscale=-1]
    \foreach \x in {0,1,2,3,4,5,6} {
      \foreach \y in {0,1,2,3,4,5,6} {
        \filldraw[black!11] (\x,\y) -- +(0,1) -- +(1,1);
        \filldraw[black!1 ] (\x,\y) -- +(1,0) -- +(1,1);
      };
    };
    \fill[black!21] (0,5) rectangle (1,6);
    \fill[black!21] (1,3) rectangle (3,5);
    \fill[black!21] (3,1) rectangle (4,3);
    \foreach \x in {1,2,3,4,5,6} {
      \draw[gray] (0,\x) -- (7,\x);
      \draw[gray] (\x,0) -- (\x,7);
    }
    \foreach \x in {(1,5),(2,4),(3,3),(4,1),(5,6),(6,2)} {\fill[black] \x circle (5pt);}
    \foreach \x in {(1,5),(3,3),(4,1),(5,6),(6,2)} {\draw \x circle (10pt);}
    \draw[thick] (0,5) rectangle (1,6);
    \draw[thick] (1,3) rectangle (3,5);
    \draw[thick] (3,1) rectangle (4,3);
  \end{tikzpicture}
  \qquad\;
  \begin{tikzpicture}[scale=.4, xscale=-1]
    \foreach \x in {0,1,2,3,4,5,6} {
      \foreach \y in {0,1,2,3,4,5,6} {
        \filldraw[black!11] (\x,\y) -- +(0,1) -- +(1,1);
        \filldraw[black!1 ] (\x,\y) -- +(1,0) -- +(1,1);
      };
    };
    \fill[black!21] (0,4) rectangle (2,6);
    \fill[black!21] (2,3) rectangle (3,4);
    \fill[black!21] (3,1) rectangle (4,3);
    \foreach \x in {1,2,3,4,5,6} {
      \draw[gray] (0,\x) -- (7,\x);
      \draw[gray] (\x,0) -- (\x,7);
    }
    \foreach \x in {(1,5),(2,4),(3,3),(4,1),(5,6),(6,2)} {\fill[black] \x circle (5pt);}
    \foreach \x in {(2,4),(3,3),(4,1),(5,6),(6,2)} {\draw \x circle (10pt);}
    \draw[thick] (0,4) rectangle (2,6);
    \draw[thick] (2,3) rectangle (3,4);
    \draw[thick] (3,1) rectangle (4,3);
  \end{tikzpicture}
  $$
  On the other hand, because $(4,3) \not\in R'$ the occurrence $26135$
  of $p$ in $q$ does not share this property:
  $$
  \begin{tikzpicture}[scale=.4, xscale=-1]
    \foreach \x in {0,1,2,3,4,5,6} {
      \foreach \y in {0,1,2,3,4,5,6} {
        \filldraw[black!11] (\x,\y) -- +(0,1) -- +(1,1);
        \filldraw[black!1 ] (\x,\y) -- +(1,0) -- +(1,1);
      };
    };
    \fill[white] (1,3) rectangle (3,5);
    \foreach \x in {1,2,3,4,5,6} {
      \draw[gray] (0,\x) -- (7,\x);
      \draw[gray] (\x,0) -- (\x,7);
    }
    \foreach \x in {(1,5),(2,4),(3,3),(4,1),(5,6),(6,2)} {
      \fill[black] \x circle (5pt);
    }
    \foreach \x in {(1,5),(3,3),(4,1),(5,6),(6,2)} {
      \draw \x circle (10pt);
    }
    \draw[thick] (1,3) rectangle (3,5);
  \end{tikzpicture}
  $$
  Therefore $q \in \av(\pi)$ and $q \in \cont(\pi')$.
\end{ex}

Lemmas~\ref{lem:p=p'} and~\ref{lem:same enc diag} give necessary
requirements for coincidence of mesh patterns. In Section~\ref{sec:enc
  diag enough}, we explore circumstances under which these conditions are
also sufficient, and in Section~\ref{sec:enc diag not enough}, we discuss
others where they are not.

\section{Pattern families that are governed by enclosed diagonals}\label{sec:enc diag enough}

There are some families of mesh patterns whose coincidence is entirely
characterized by the collection of enclosed diagonals.

\begin{defn}
  A mesh pattern $(p,R)$ is \emph{vincular} if the mesh $R$ is a union
  of complete columns. This $(p,R)$ can also be described as the
  permutation $p \in \mf{S}_k$ together with ``bonds'' between $p(i)$
  and $p(i+1)$ exactly when $\{(i,y) : 0 \le y \le k\} \subseteq
  R$. Thus a permutation $w$ contains this vincular pattern if and only
  if $w$ has an occurrence of $p$ in which all bonded subwords appear
  consecutively in $w$.
\end{defn}

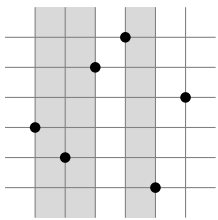
\begin{figure}[htbp]
  $$
  \begin{tikzpicture}[scale=.4]
    \fill[lightgray] (1,0) rectangle (3,7);
    \fill[lightgray] (4,0) rectangle (5,7);
    \foreach \x in {1,2,3,4,5,6} {
      \draw[gray] (\x,0) -- (\x,7);
      \draw[gray] (0,\x) -- (7,\x);
    }
    \foreach \x in {(1,3),(2,2),(3,5),(4,6),(5,1),(6,4)} {
      \fill[black] \x circle (5pt);
    }
  \end{tikzpicture}
  $$
  \caption{A vincular pattern.}\label{fig:vincular}
\end{figure}

The mesh pattern in Figure~\ref{fig:vincular} is vincular
because its mesh is a union of columns.
It could also be denoted
$\underbracket[.5pt][1pt]{325}\underbracket[.5pt][1pt]{61}\!4$, with
brackets indicating the bonds.
Vincular patterns have arisen in many places, including
\cite{babson-steingrimsson, claesson-01, steingrimsson, tenner-coincidental}. Here,
we show that coincidence within this family of mesh patterns behaves
exactly as one might hope.

\begin{prop}\label{prop:vincular}
  Let $(p,R)$ and $(p,R')$ be vincular patterns where $p \in \mathfrak{S}_k$. If
  $\enc{(p,R)} = \enc{(p,R')}$, then $(p,R) \asymp (p,R')$.  If, moreover,
  $k > 3$, then $R = R'$ and so the two patterns are the same.
\end{prop}

\begin{proof}
  Suppose that $(p,R)$ and $(p,R')$ are vincular patterns with $p \in
  \mf{S}_k$ and $\enc{(p,R)} = \enc{(p,R')}$. The result is easy to
  check for $k \le 3$, so suppose $k > 3$. Each column in the meshes $R$
  and $R'$ contains at least $k+1 - 4 > 0$ pointless squares, and each of those
  pointless squares is an enclosed diagonal. Thus we have, in fact, that the
  meshes $R$ and $R'$ are comprised of exactly the same columns and so
  $R = R'$. Coincidence is a reflexive relation, completing the proof.
\end{proof}

The example in Figure~\ref{fig:2-3-1-and-2-31} (see \cite[Lemma 2]{claesson-01})
shows that when $p \in \mf{S}_k$ and $k \le 3$, there are indeed pairs
of distinct coincidental vincular patterns.

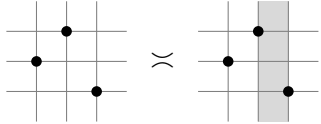
\begin{figure}[htbp]
  $$
  \begin{tikzpicture}[scale=.4, baseline={([yshift=-2.5pt]current bounding box.center)}]
    \foreach \x in {1,2,3} {
      \draw[gray] (\x,0) -- (\x,4);
      \draw[gray] (0,\x) -- (4,\x);
    }
    \foreach \x in {(1,2),(2,3),(3,1)} {\fill[black] \x circle (5pt);}
  \end{tikzpicture}
  \;\,\asymp\,\;
  \begin{tikzpicture}[scale=.4, baseline={([yshift=-2.5pt]current bounding box.center)}]
    \fill[lightgray] (2,0) rectangle (3,4);
    \foreach \x in {1,2,3} {
      \draw[gray] (\x,0) -- (\x,4);
      \draw[gray] (0,\x) -- (4,\x);
    }
    \foreach \x in {(1,2),(2,3),(3,1)} {\fill[black] \x circle (5pt);}
  \end{tikzpicture}
  $$
  \caption{Two distinct but coincidental vincular patterns.}\label{fig:2-3-1-and-2-31}
\end{figure}

\begin{cor}
  Two vincular patterns $(p,R)$ and $(p',R')$ are coincident if and only
  if $p=p'$ and $\enc{(p,R)} = \enc{(p',R')}$.  Moreover, if $p \in
  \mf{S}_k$ for $k > 3$, then $(p,R) \asymp (p',R')$ if and only if
  $(p,R) = (p',R')$.
\end{cor}

\begin{proof}
  This follows from Proposition~\ref{prop:vincular} together with Lemmas~\ref{lem:p=p'}
  and~\ref{lem:same enc diag}.
\end{proof}

Vincularity is a property of the mesh. The next family we examine can
also be defined by its mesh, this time requiring that the mesh be
suitably meager.

\begin{defn}
  A mesh pattern $(p,R)$ is \emph{isolating} if whenever $(i,j) \in R$
  is not a pointless square, then the mesh $R$ contains no elements of
  the form $(i\pm1,y)$ or $(x,j\pm1)$.
\end{defn}

As with vincular patterns, coincidence among isolating patterns is
determined entirely by the set of enclosed diagonals in the isolating
pattern.

\begin{prop}
  If $(p,R)$ and $(p,R')$ are isolating mesh patterns with $\enc{(p,R)}
  = \enc{(p,R')}$, then $(p,R) \asymp (p,R')$.
\end{prop}

\begin{proof}
  The Shading Lemma of~\cite{wilfshort} applies directly to show that
  such isolating patterns are coincident.
\end{proof}

\section{Pattern families that are not governed by enclosed diagonals}\label{sec:enc diag not enough}

In this section we look at some families of mesh patterns for which
common enclosed diagonals are not enough to determine pattern
coincidence. Despite this fact, these families are not entirely
unrelated to those discussed in Section~\ref{sec:enc diag enough}.

Recall the definition of vincular patterns from Section~\ref{sec:enc
  diag enough}. This placed a requirement on the columns of a mesh but
not on the rows, suggesting a certain asymmetry. Indeed, as studied in
\cite{bousquet-melou claesson dukes kitaev}, a more balanced definition
yields another interesting class of permutation patterns.

\begin{defn}
  \emph{Bivincular} patterns are mesh patterns in which the mesh is a
  union of complete rows and complete columns.
\end{defn}

Given their similarity to vincular patterns, one might hope that
coincidence among bivincular patterns can also be fully characterized by
enclosed diagonals. This is not the case, however, as demonstrated below.

\begin{ex}\label{ex:bivincular}\
  Up to symmetry, the smallest non-coincident bivincular patterns
  having the same enclosed diagonals are
  $$
  
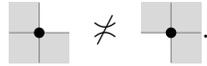
\begin{tikzpicture}[scale=.4, baseline=9.6pt]
    \foreach \x in {(0,0),(0,1),(1,0)} {
      \fill[lightgray] \x rectangle ++(1,1);
    }
    \foreach \x in {1} {
      \draw[gray] (\x,0) -- (\x,2);
      \draw[gray] (0,\x) -- (2,\x);
    }
    \foreach \x in {(1,1)} {\fill[black] \x circle (5pt);}
  \end{tikzpicture}
  \;\,\not\asymp\,\;
  \begin{tikzpicture}[scale=.4, baseline=9.6pt]
    \foreach \x in {(1,1),(0,1),(1,0)} {\fill[lightgray] \x rectangle ++(1,1);}
    \foreach \x in {1} {
      \draw[gray] (\x,0) -- (\x,2);
      \draw[gray] (0,\x) -- (2,\x);
    }
    \foreach \x in {(1,1)} {\fill[black] \x circle (5pt);}
  \end{tikzpicture}.
  $$
  The unique enclosed diagonal in each case is $\{(0,1),(1,0)\}$, but
  the patterns are not coincident because $132$ contains the first
  pattern and avoids the second.  For a larger example see
  Figure~\ref{fig:non-coincident-bivincular}.
  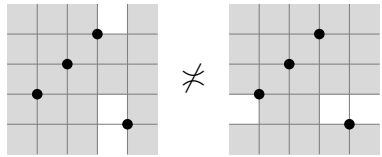
\begin{figure}[htbp]
    $$
    \begin{tikzpicture}[scale=.4, baseline=26pt]
      \foreach \x in {(0,0),(0,1),(0,2),(0,3),(0,4),(1,0),(1,1),
                      (1,2),(1,3),(1,4),(2,0),(2,1),(2,2),(2,3),
                      (2,4),(3,2),(3,3),(4,0),(4,1),(4,2),(4,3),(4,4)} {
        \fill[lightgray] \x rectangle ++(1,1);
      }
      \foreach \x in {1,2,3,4} {
        \draw[gray] (\x,0) -- (\x,5);
        \draw[gray] (0,\x) -- (5,\x);
      }
      \foreach \x in {(1,2),(2,3),(3,4),(4,1)} {\fill[black] \x circle (5pt);}
    \end{tikzpicture}
    \;\,\not\asymp\,\;
    \begin{tikzpicture}[scale=.4, baseline=26pt]
      \foreach \x in {(0,0),(0,2),(0,3),(0,4),(1,0),(1,1),(1,2),(1,3),
                      (1,4),(2,0),(2,1),(2,2),(2,3),(2,4),(3,0),(3,2),
                      (3,3),(3,4),(4,0),(4,2),(4,3),(4,4)} {
        \fill[lightgray] \x rectangle ++(1,1);
      }
      \foreach \x in {1,2,3,4} {
        \draw[gray] (\x,0) -- (\x,5);
        \draw[gray] (0,\x) -- (5,\x);
      }
      \foreach \x in {(1,2),(2,3),(3,4),(4,1)} {\fill[black] \x circle (5pt);}
    \end{tikzpicture}
    $$
    \caption{Two non-coincident bivincular patterns with the same
      enclosed diagonals.}\label{fig:non-coincident-bivincular}
  \end{figure}
  The proper enclosed diagonals in each case are $\{(0,2), (1,1)\},
  \{(1,3), (2,2)\}$ and $\{(2,4), (3,3)\}$, and they have the same
  eleven pointless squares, but the patterns are not coincident because
  $345162$ contains the first pattern and avoids the second.
\end{ex}

Recall the definition of isolating patterns from Section~\ref{sec:enc
  diag enough}, and consider the following class of patterns that might,
initially, appear to be similarly desolate.

\begin{defn}
  A mesh pattern $(p,R)$ is \emph{sparse} if it has at most one shaded
  square in each row and each column.
\end{defn}

Unfortunately, coincidence within this family of mesh patterns is not
solely dependent on the enclosed diagonals in the pattern. That is,
there are non-coincident sparse mesh patterns that have the same
enclosed diagonals.

\begin{ex}
  The mesh patterns $\pi = (231, \{ (3,2) \})$ and $\sigma = (231, \{
  (1,3), (3,2) \})$ are both sparse. Further, they satisfy $\enc{(\pi)}
  = \enc{(\sigma)} = \{(3,2)\}$, but $\pi \not\asymp \sigma$ because
  $25314$ contains $\pi$ but avoids $\sigma$.
\end{ex}

One might hope that nice geometric properties of the mesh would be
enough to ensure that enclosed diagonals determine
coincidence. Unfortunately, even a basic property like mesh connectivity
is insufficient, as demonstrated in Example~\ref{ex:bivincular} above.

\section{The simultaneous shading lemma}\label{sec:ssl}

We now give an extension of the Shading Lemma
(Lemma~\ref{lem:shading} above), the conditions of which are depicted in
Figure~\ref{fig:doubleshading}.

\begin{cor}[Double Shading Lemma (east)] \label{cor:shadinglemmaextension}
  Let $(p,R)$ be a mesh pattern. Let $i$ be such that the following
  conditions all hold.
  \begin{itemize}
  \item No square incident to $(i,p(i)) \in G(p)$ is in the mesh $R$.
  \item For all $x$, $\boks{x}{p(i)-1} \in R$ if and only if
  $\boks{x}{p(i)} \in R$.
  \item For all $y$, if $\boks{i-1}{y} \in R$ then $\boks{i}{y} \in R$.
  \end{itemize}
  Then $(p,R) \asymp \big(p,R \cup \{(i,p(i)), (i,p(i)-1)\}\big)$.
\end{cor}

\begin{figure}[htbp]
  $$
  \begin{tikzpicture}[scale=1]
    \filldraw[fill=lightgray,draw=gray] (1,0) rectangle (3,1);
    \filldraw[fill=lightgray,draw=gray] (0,1) rectangle (1,3);
    \filldraw[fill=lightgray,draw=gray] (1,3) rectangle (3,4);
    \filldraw[fill=lightgray,draw=gray] (3,1) rectangle (4,3);
    \foreach \x in {1,2,3} {
      \draw[gray] (\x,-0.5) -- (\x,4.5);
      \draw[gray] (-0.5,\x) -- (4.5,\x);
    }
    \draw[->, double] (1.5,0.5) -- (2.5,0.5);
    \draw[<->, double] (0.5,1.5) -- (0.5,2.5);
    \draw[->, double] (1.5,3.5) -- (2.5,3.5);
    \draw[<->, double] (3.5,1.5) -- (3.5,2.5);
    \fill[black] (2,2) circle (4pt);
  \end{tikzpicture}
  $$
  \caption{The conditions required to apply
    Corollary~\ref{cor:shadinglemmaextension}, the Double Shading Lemma
    (east).}\label{fig:doubleshading}
\end{figure}

\begin{proof}
  First use the northeast variant of the Shading Lemma on $(i,p(i))$ to
  establish that $(p,R) \asymp (p,R \cup \{(i,p(i))\})$.  It is easy to
  see that we can then apply the southeast variant of the Shading Lemma
  to $(i,p(i)-1)$ to get
  \[
  \big(p,R \cup \{(i,p(i))\}\big) \asymp
  \big(p,R \cup \{(i,p(i)), (i,p(i)-1)\}\big).\qedhere
  \]
\end{proof}

Corollary~\ref{cor:shadinglemmaextension} refers to the two squares
incident to an element of $G(p)$ on its east side. There are, of course,
analogous statements for the north, west, and south sides of elements of
$G(p)$.

It is interesting to note that there are mesh pattern coincidences that
are not detected by the Shading and Double Shading Lemmas.

\begin{ex}\label{ex:surprising}
  Corollary~\ref{cor:shadinglemmaextension} yields
  $$\pi_1 = (12,R_1) \;=\;
  \begin{tikzpicture}[scale=.4, baseline={([yshift=-2.5pt]current bounding box.center)}]
    \foreach \x in {(2,0)} {
      \fill[lightgray] \x rectangle ++ (1,1);
    }
    \foreach \x in {1,2} {
      \draw[gray] (0,\x) -- (3,\x);
      \draw[gray] (\x,0) -- (\x,3);
    }
    \foreach \x in {(1,1),(2,2)} {
      \fill[black] \x circle (5pt);
    }
  \end{tikzpicture}
  \;\,\asymp\,\;
  \begin{tikzpicture}[scale=.4, baseline={([yshift=-2.5pt]current bounding box.center)}]
    \foreach \x in {(2,0),(1,0),(0,0)} {\fill[lightgray] \x rectangle ++ (1,1);}
    \foreach \x in {1,2} {
      \draw[gray] (0,\x) -- (3,\x);
      \draw[gray] (\x,0) -- (\x,3);
    }
    \foreach \x in {(1,1),(2,2)} {\fill[black] \x circle (5pt);}
  \end{tikzpicture}
  \;\,\asymp\,\;
  \begin{tikzpicture}[scale=.4, baseline={([yshift=-2.5pt]current bounding box.center)}]
    \foreach \x in {(2,0),(1,0),(0,0),(1,1),(1,2)} {\fill[lightgray] \x rectangle ++ (1,1);}
    \foreach \x in {1,2} {
      \draw[gray] (0,\x) -- (3,\x);
      \draw[gray] (\x,0) -- (\x,3);
    }
    \foreach \x in {(1,1),(2,2)} {\fill[black] \x circle (5pt);}
  \end{tikzpicture}
  \;\,=\,\; (12,R_2) = \pi_2.
  $$
  It is, however, impossible to use Lemma~\ref{lem:shading} or
  Corollary~\ref{cor:shadinglemmaextension} to prove that any of the
  above patterns are coincident to
  $$\pi_3 = (12,R_3) \,=\,
  \begin{tikzpicture}[scale=.4, baseline={([yshift=-2.5pt]current bounding box.center)}]
    \foreach \x in {(0,0),(1,1),(2,0)} {
      \fill[lightgray] \x rectangle ++ (1,1);
    }
    \foreach \x in {1,2} {
      \draw[gray] (0,\x) -- (3,\x);
      \draw[gray] (\x,0) -- (\x,3);
    }
    \foreach \x in {(1,1),(2,2)} {
      \fill[black] \x circle (5pt);
    }
  \end{tikzpicture}\,.
  $$
  On the other hand, the mesh inclusions $R_1 \subset R_3 \subset R_2$
  imply that $\av(\pi_1) \subseteq \av(\pi_3) \subseteq \av(\pi_2)$.
  Since $\av(\pi_1) = \av(\pi_2)$, we must have that $\pi_3$ is, in
  fact, coincident with all three of the above mesh patterns.
\end{ex}

The following lemma records the phenomenon observed in
Example~\ref{ex:surprising}.

\begin{lem}[Closure Lemma] \label{lem:closure}
  If $(p,R) \asymp (p,R')$
  and $R \subseteq S \subseteq R'$, then
  $$(p,R) \asymp (p,S) \asymp (p,R').
  $$
\end{lem}

We now extend the Shading Lemma even further to allow simultaneous
addition of multiple squares to a mesh. Before stating this precisely, we
illustrate the concept.

\begin{ex}\label{ex:simultaneous-shading}
  Consider the four mesh patterns in Figure~\ref{fig:simultaneous-shading}.
  Applying the Shading Lemma (northwest) to $(231,R)$ at $(1,2) \in
  G(231)$ gives the coincident pattern $(231,R_1)$. On the other hand,
  applying the Double Shading Lemma (north) to $(231,R)$ at $(2,3) \in
  G(231)$ yields the coincident pattern $(231,R_2)$.  It is easy to
  verify directly that these patterns are all coincident to $(231,R_3)$
  even though iterations of the Shading and Double Shading Lemmas will
  not produce this coincidence. For example, we cannot apply the Double
  Shading Lemma (north) to $(231,R_1)$ at $(1,2)$ because of $(0,3) \in
  R_1$. Similarly we cannot apply the Shading Lemma (northwest) to
  $(231,R_2)$ at $(1,2)$ because of $(1,3) \in R_2$.
\end{ex}

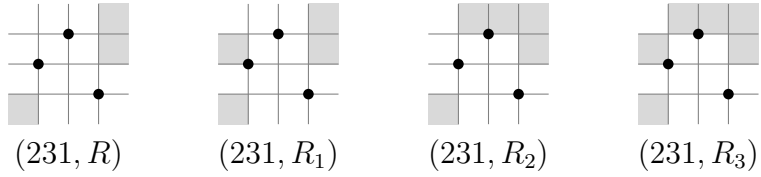
\begin{figure}[htbp]
  \[
  \begin{array}{c}
    \begin{tikzpicture}[scale=.4, baseline={([yshift=-2.5pt]current bounding box.center)}]
      \foreach \x in {(0,0),(3,2),(3,3)} {\fill[lightgray] \x rectangle ++ (1,1);}
      \foreach \x in {1,2,3} {
        \draw[gray] (0,\x) -- (4,\x);
        \draw[gray] (\x,0) -- (\x,4);
      }
      \foreach \x in {(1,2),(2,3),(3,1)} {\fill[black] \x circle (5pt);}
    \end{tikzpicture}\medskip\\
    (231,R)
  \end{array}
  \qquad
  \begin{array}{c}
    \begin{tikzpicture}[scale=.4, baseline={([yshift=-2.5pt]current bounding box.center)}]
      \foreach \x in {(0,0),(3,2),(3,3),(0,2)} {\fill[lightgray] \x rectangle ++ (1,1);}
      \foreach \x in {1,2,3} {
        \draw[gray] (0,\x) -- (4,\x);
        \draw[gray] (\x,0) -- (\x,4);
      }
      \foreach \x in {(1,2),(2,3),(3,1)} {\fill[black] \x circle (5pt);}
    \end{tikzpicture}\medskip\\
    (231,R_1)
  \end{array}
  \qquad
  \begin{array}{c}
    \begin{tikzpicture}[scale=.4, baseline={([yshift=-2.5pt]current bounding box.center)}]
      \foreach \x in {(0,0),(3,2),(3,3),(1,3),(2,3)} {\fill[lightgray] \x rectangle ++ (1,1);}
      \foreach \x in {1,2,3} {
        \draw[gray] (0,\x) -- (4,\x);
        \draw[gray] (\x,0) -- (\x,4);
      }
      \foreach \x in {(1,2),(2,3),(3,1)} {\fill[black] \x circle (5pt);}
    \end{tikzpicture}\medskip\\
     (231,R_2)
   \end{array}
  \qquad
  \begin{array}{c}
    \begin{tikzpicture}[scale=.4, baseline={([yshift=-2.5pt]current bounding box.center)}]
      \foreach \x in {(0,0),(3,2),(3,3),(0,2),(1,3),(2,3)} {\fill[lightgray] \x rectangle ++ (1,1);}
      \foreach \x in {1,2,3} {
        \draw[gray] (0,\x) -- (4,\x);
        \draw[gray] (\x,0) -- (\x,4);
      }
      \foreach \x in {(1,2),(2,3),(3,1)} {\fill[black] \x circle (5pt);}
    \end{tikzpicture}\medskip\\
    (231,R_3)
  \end{array}
  \]
  \caption{The four mesh patterns of
    Example~\ref{ex:simultaneous-shading}.}\label{fig:simultaneous-shading}
\end{figure}

\begin{defn}
  Let $(p,R)$ be a mesh pattern. A square that is incident to $g \in
  G(p)$ but not in $R$ is \emph{shadeable from $g$} if it satisfies the
  conditions of one of the Shading Lemma variants. A pair of adjacent
  squares that are both incident to $g \in G(p)$ and both not in $R$ are
  \emph{shadeable from $g$} if they satisfy the conditions of one of the
  Double Shading Lemma variants.
\end{defn}

\begin{lem}[Simultaneous Shading Lemma]\label{lem:simult-shading}
  Let $(p,R)$ be a mesh pattern with $p \in \mf{S}_k$. Fix $G \subseteq
  G(p)$ and let $s_g$ be a square or pair of adjacent squares that are
  shadeable from $g \in G$.  Then
  $(p,R) \asymp \left(p,R \cup S\right)$, where $S = \bigcup_{g \in G} s_g$.
\end{lem}

Before giving a proof of Lemma~\ref{lem:simult-shading} we illustrate
its idea in the following example.

\begin{ex}\label{ex:simult-shading}
  Let the two mesh patterns $(p,R)$ and $(p,R')$ be as in
  Figure~\ref{fig:simult-shading}. Let $G = \{g_1,g_2,g_3\} \subset
  G(p)$ where $g_i = (i,p(i))$. If we let $s_1 = \{ (0,0), (1,0) \}$,
  $s_2 = \{ (1,4), (2,4) \}$, and $s_3 = \{(3,1)\}$, then $R' = R \cup
  s_1 \cup s_2 \cup s_3$. Note that, for $i\in \{1,2,3\}$, the square
  $s_i$ is shadeable from $g_i$. The Simultaneous Shading Lemma claims
  that $(p,R) \asymp (p,R')$.

  \begin{figure}[htbp]
  $$
  (p,R) =
  \begin{tikzpicture}[scale=.4, baseline={([yshift=-2.5pt]current bounding box.center)}]
    \foreach \x in {(4,0),(4,1)} {\fill[lightgray] \x rectangle ++ (1,1);}
    \foreach \x in {1,2,3,4} {
      \draw[gray] (0,\x) -- (5,\x);
      \draw[gray] (\x,0) -- (\x,5);
    }
    \foreach \x in {(1,1),(2,4),(3,2),(4,3)} {\fill[black] \x circle (5pt);}
  \end{tikzpicture}
  \qquad\quad
  (p,R') =
  \begin{tikzpicture}[scale=.4, baseline={([yshift=-2.5pt]current bounding box.center)}]
    \foreach \x in {(0,0),(1,0),(1,4),(2,4),(3,1),(4,0),(4,1)} {
      \fill[lightgray] \x rectangle ++ (1,1);
    }
    \foreach \x in {1,2,3,4} {
      \draw[gray] (0,\x) -- (5,\x);
      \draw[gray] (\x,0) -- (\x,5);
    }
    \foreach \x in {(1,1),(2,4),(3,2),(4,3)} {\fill[black] \x circle (5pt);}
  \end{tikzpicture}
  $$
  \caption{The two coincident mesh patterns of
    Example~\ref{ex:simult-shading}.}\label{fig:simult-shading}
\end{figure}
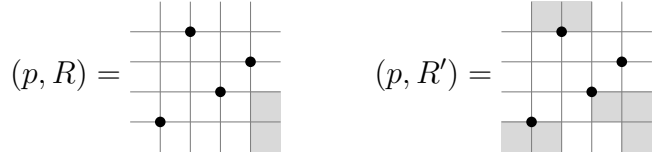

  We wish to show that any permutation containing $(p,R)$ also contains
  $(p,R')$; since $R\subset R'$ we already know that
  $\cont(p,R) \supseteq \cont(p,R')$.
  Let $w = 482951(10)376 \in \mf{S}_{10}$ and consider the occurrence
  $4857$ of $(p,R)$ in $w$. For this to be an occurrence of $(p,R')$ in
  $w$, we would need to be able to shade the regions corresponding to
  $s_1 \cup s_2 \cup s_3$ without intersecting elements of
  $G(w)$. Unfortunately $(4,9) \in G(w)$ intersects the region $[1,5]
  \times [8,11]$ corresponding to $s_2$. If we consider, instead, the
  occurrence $4957$ of $p$ in $w$, then we would need to worry about
  $(3,2) \in G(w)$ intersecting the region $[0,0] \times [4,4]$
  corresponding to $s_1$. So let us, instead, consider the occurrence
  $2957$ of $p$ in $w$. Now we must worry about $(7,3) \in G(w)$
  intersecting the region $[5,9]\times[2,5]$ corresponding to $s_3$. We
  might then consider the occurrence $2937$ of $p$ in $w$, and so
  on. This procedure is illustrated in Figure~\ref{fig:sslexamplecont},
  where dashed rectangles indicate the regions corresponding to the
  $s_i$ and arrows indicate the ensuing re-selection of the pattern
  occurrence. The last image in the figure depicts an occurrence of
  $(p,R')$ in $w$.
\end{ex}

\begin{center}
  \begin{figure}[ht]
    $$
    \begin{array}{cc}
      \begin{tikzpicture}[scale=0.37]
        \foreach \x in {(9,0)} {\fill[lightgray] \x rectangle ++ (2,5);}
        \foreach \x in {1,2,3,4,5,6,7,8,9,10} {
          \draw[gray] (0,\x) -- (11,\x);
          \draw[gray] (\x,0) -- (\x,11);
        }
        \draw[thick, white] (0,0) rectangle (2,4);  
        \draw[thick, white] (1,8) rectangle (5,11); 
        \draw[thick, white] (5,4) rectangle (9,5);  
        \draw[thick, dashed] (0,0) rectangle (2,4);  
        \draw[thick, dashed] (1,8) rectangle (5,11); 
        \draw[thick, dashed] (5,4) rectangle (9,5);  
        \draw[thick] (9,0) rectangle (11,4); 
        \draw[thick] (9,4) rectangle (11,5); 
        \foreach \x in {(1,4),(2,8),(3,2),(4,9),(5,5),(6,1),(7,10),(8,3),(9,7),(10,6)} {
          \fill[black] \x circle (5pt);
        }
        \foreach \x in {(1,4),(2,8),(5,5),(9,7)} {\draw \x circle (10pt);}
        \draw[thick, ->] (2,8) -- (4-0.25,9-0.25);   
      \end{tikzpicture}
      \quad&\quad
      \begin{tikzpicture}[scale=0.37]
        \foreach \x in {(9,0)} {\fill[lightgray] \x rectangle ++ (2,5);}
        \foreach \x in {1,2,3,4,5,6,7,8,9,10} {
          \draw[gray] (0,\x) -- (11,\x);
          \draw[gray] (\x,0) -- (\x,11);
        }
        \draw[thick, white] (0,0) -- (4,0) -- (4,4) -- (0,4) -- cycle;   
        \draw[thick, white] (1,9) -- (5,9) -- (5,11) -- (1,11) -- cycle; 
        \draw[thick, white] (5,4) -- (9,4) -- (9,5) -- (5,5) -- cycle;   
        \draw[thick, dashed] (0,0) -- (4,0) -- (4,4) -- (0,4) -- cycle;   
        \draw[thick, dashed] (1,9) -- (5,9) -- (5,11) -- (1,11) -- cycle; 
        \draw[thick, dashed] (5,4) -- (9,4) -- (9,5) -- (5,5) -- cycle;   
        \draw[thick] (9,0) rectangle (11,4); 
        \draw[thick] (9,4) rectangle (11,5); 
        \foreach \x in {(1,4),(2,8),(3,2),(4,9),(5,5),(6,1),(7,10),(8,3),(9,7),(10,6)} {
          \fill[black] \x circle (5pt);
        }
        \foreach \x in {(1,4),(4,9),(5,5),(9,7)} {\draw \x circle (10pt);}
        \draw[thick, ->] (1,4) -- (3-0.25,2+0.25);   
      \end{tikzpicture}\smallskip\\
      \text{(a)} \quad&\quad \text{(b)} \\[3ex]
      \begin{tikzpicture}[scale=0.37]
        \foreach \x in {(9,0)} {\fill[lightgray] \x rectangle ++ (2,5);}
        \foreach \x in {1,2,3,4,5,6,7,8,9,10} {
          \draw[gray] (0,\x) -- (11,\x);
          \draw[gray] (\x,0) -- (\x,11);
        }
        \draw[thick, white] (0,0) -- (4,0) -- (4,2) -- (0,2) -- cycle;   
        \draw[thick, white] (3,9) -- (5,9) -- (5,11) -- (3,11) -- cycle; 
        \draw[thick, white] (5,2) -- (9,2) -- (9,5) -- (5,5) -- cycle;   
        \draw[thick, dashed] (0,0) -- (4,0) -- (4,2) -- (0,2) -- cycle;   
        \draw[thick, dashed] (3,9) -- (5,9) -- (5,11) -- (3,11) -- cycle; 
        \draw[thick, dashed] (5,2) -- (9,2) -- (9,5) -- (5,5) -- cycle;   
        \draw[thick] (9,0) rectangle (11,2); 
        \draw[thick] (9,2) rectangle (11,5); 
        \foreach \x in {(1,4),(2,8),(3,2),(4,9),(5,5),(6,1),(7,10),(8,3),(9,7),(10,6)} {
          \fill[black] \x circle (5pt);
        }
        \foreach \x in {(3,2),(4,9),(5,5),(9,7)} {\draw \x circle (10pt);}
        \draw[thick, ->] (5,5) -- (8-0.25,3+0.25);   
      \end{tikzpicture}
      \quad&\quad
      \begin{tikzpicture}[scale=0.37]
        \foreach \x in {(9,0)} {\fill[lightgray] \x rectangle ++ (2,3);}
        \foreach \x in {1,2,3,4,5,6,7,8,9,10} {
          \draw[gray] (0,\x) -- (11,\x);
          \draw[gray] (\x,0) -- (\x,11);
        }
        \draw[thick, white] (0,0) -- (4,0) -- (4,2) -- (0,2) -- cycle;   
        \draw[thick, white] (3,9) -- (8,9) -- (8,11) -- (3,11) -- cycle; 
        \draw[thick, white] (8,2) -- (9,2) -- (9,3) -- (8,3) -- cycle;   
        \draw[thick, dashed] (0,0) -- (4,0) -- (4,2) -- (0,2) -- cycle;   
        \draw[thick, dashed] (3,9) -- (8,9) -- (8,11) -- (3,11) -- cycle; 
        \draw[thick, dashed] (8,2) -- (9,2) -- (9,3) -- (8,3) -- cycle;   
        \draw[thick] (9,0) rectangle (11,2); 
        \draw[thick] (9,2) rectangle (11,3); 
        \foreach \x in {(1,4),(2,8),(3,2),(4,9),(5,5),(6,1),(7,10),(8,3),(9,7),(10,6)} {
          \fill[black] \x circle (5pt);
        }
        \foreach \x in {(3,2),(4,9),(8,3),(9,7)} {\draw \x circle (10pt);}
        \draw[thick, ->] (4,9) -- (7-0.25,10-0.25);   
      \end{tikzpicture}\smallskip\\
      \text{(c)} \quad&\quad \text{(d)} \\[3ex]
      \begin{tikzpicture}[scale=0.37]
        \foreach \x in {(9,0)} {\fill[lightgray] \x rectangle ++ (2,3);}
        \foreach \x in {1,2,3,4,5,6,7,8,9,10} {
          \draw[gray] (0,\x) -- (11,\x);
          \draw[gray] (\x,0) -- (\x,11);
        }
        \draw[thick, white] (0,0) -- (7,0) -- (7,2) -- (0,2) -- cycle;      
        \draw[thick, white] (3,10) -- (8,10) -- (8,11) -- (3,11) -- cycle; 
        \draw[thick, white] (8,2) -- (9,2) -- (9,3) -- (8,3) -- cycle;      
        \draw[thick, dashed] (0,0) -- (7,0) -- (7,2) -- (0,2) -- cycle;      
        \draw[thick, dashed] (3,10) -- (8,10) -- (8,11) -- (3,11) -- cycle; 
        \draw[thick, dashed] (8,2) -- (9,2) -- (9,3) -- (8,3) -- cycle;      
        \draw[thick] (9,0) rectangle (11,2); 
        \draw[thick] (9,2) rectangle (11,3); 
        \foreach \x in {(1,4),(2,8),(3,2),(4,9),(5,5),(6,1),(7,10),(8,3),(9,7),(10,6)} {
          \fill[black] \x circle (5pt);
        }
        \foreach \x in {(3,2),(7,10),(8,3),(9,7)} {\draw \x circle (10pt);}
        \draw[thick, ->] (3,2) -- (6-0.25,1+0.25);   
      \end{tikzpicture}
      \quad&\quad
      \begin{tikzpicture}[scale=0.37]
        \foreach \x in {(9,0)} {\fill[lightgray] \x rectangle ++ (2,3);}
        \foreach \x in {(0,0)} {\fill[lightgray] \x rectangle ++ (7,1);}
        \foreach \x in {(6,10)} {\fill[lightgray] \x rectangle ++ (2,1);}
        \foreach \x in {(8,1)} {\fill[lightgray] \x rectangle ++ (1,2);}
        \foreach \x in {1,2,3,4,5,6,7,8,9,10} {
          \draw[gray] (0,\x) -- (11,\x);
          \draw[gray] (\x,0) -- (\x,11);
        }
        \foreach \x in {(1,4),(2,8),(3,2),(4,9),(5,5),(6,1),(7,10),(8,3),(9,7),(10,6)} {
          \fill[black] \x circle (5pt);
        }
        \foreach \x in {(6,1),(7,10),(8,3),(9,7)} {\draw \x circle (10pt);}
        \draw[thick] (0,0) -- (7,0) -- (7,1) -- (0,1) -- cycle;     
        \draw[thick] (6,10) -- (8,10) -- (8,11) -- (6,11) -- cycle; 
        \draw[thick] (8,1) -- (9,1) -- (9,3) -- (8,3) -- cycle;     
        \draw[thick] (9,0) rectangle (11,1); 
        \draw[thick] (9,1) rectangle (11,3); 
      \end{tikzpicture}\smallskip\\
      \text{(e)} \quad&\quad \text{(f)}
    \end{array}
    $$
    \caption{Maintaining the notation of
      Example~\ref{ex:simult-shading}, these figures depict the
      discovery of an occurrence of $(p,R')$ in $w$.}
    \label{fig:sslexamplecont}
  \end{figure}
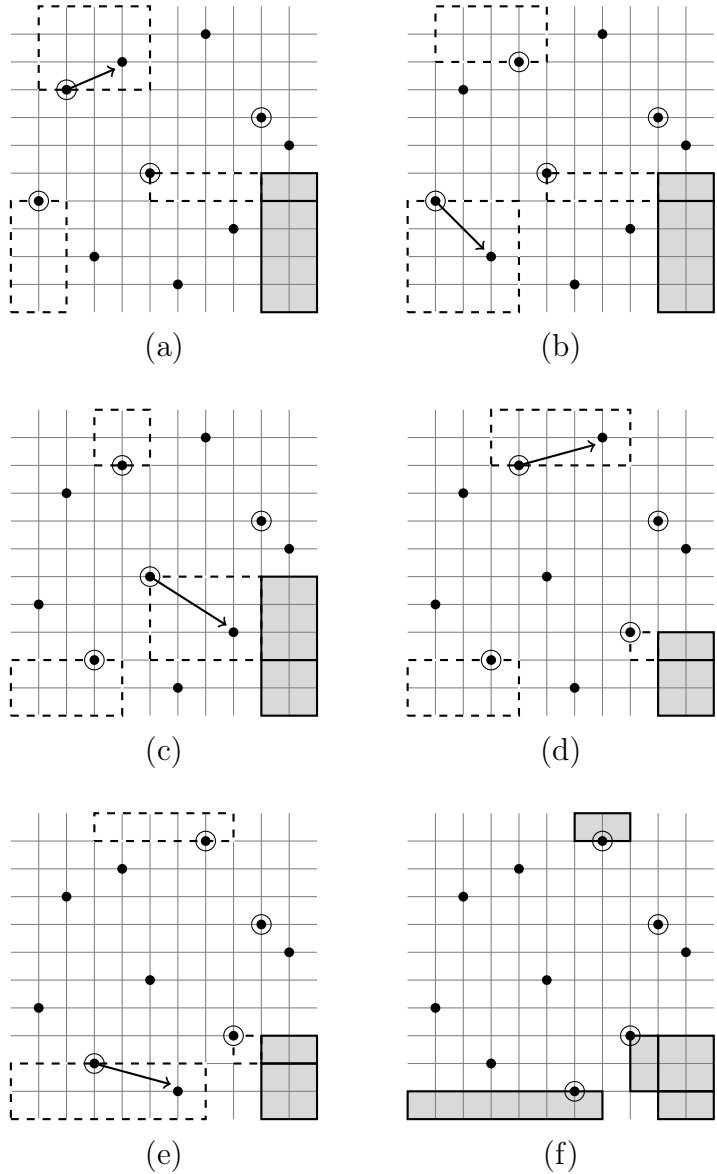
\end{center}

Example~\ref{ex:simult-shading} and Figure~\ref{fig:sslexamplecont}
describe the strategy to the proof of Lemma~\ref{lem:simult-shading},
but we need to show that the procedure terminates.

\begin{proof}[Proof of Lemma~\ref{lem:simult-shading}]
  Let $w$ be a permutation containing the mesh pattern $(p,R)$.  We want
  to show that $w$ also contains an occurrence of $(p,R')$. Fix an
  occurrence of $(p,R)$ in $w$. Each $g \in G \subseteq G(p)$
  corresponds to some $(i_g, w(i_g)) \in G(w)$ in the occurrence of
  $(p,R)$ in $w$. Together with $s_g$, this point will determine
  \begin{itemize}
  \item the horizontal line $y = w(i_g)$ if $s_g$ is a pair of
    horizontally adjacent squares,
  \item the vertical line $x = i_g$ if $s_g$ is a pair of vertically
    adjacent squares, or
  \item the pair of lines $x = i_g$ and $y = w(i_g)$ if $s_g$ is a
    single square.
  \end{itemize}
  For each $g \in G$, let $\widetilde{s}_g$ be the region in $w$ that
  corresponds to $s_g$, as defined by the given occurrence of $(p,R)$ in
  $w$. If each $\widetilde{s}_g$ is empty then we have already found an
  occurrence of $(p,R')$ in $w$. Otherwise, fix $h \in G$ so that
  $\widetilde{s}_h$ is nonempty.  There are essentially two different
  options for $s_h$:
  \begin{enumerate}
  \item\label{case-a} $s_h$ is a pair of adjacent squares;
  \item\label{case-b} $s_h$ is a single square.
  \end{enumerate}

  For case~\eqref{case-a} suppose that $s_h$ is a pair of horizontally
  adjacent squares that are north of and incident to $h \in G(p)$. Now
  replace $(i_h, w(i_h))$ in the occurrence of $(p,R)$ by the element of
  $(i'_h, w(i'_h)) \in \widetilde{s}_h \cap G(w)$ having
  the largest $y$-coordinate.  Because $s_h$ was shadeable from $h$ this
  will produce a new occurrence of $(p,R)$ in $w$.  Also, note that
  $s_h$ now determines the line $y = w(i'_h)$, which is above the
  previously identified line $y = w(i_h)$.

  For case~\eqref{case-b} suppose $s_h$ is a single square northeast of
  and incident to $h\in G(p)$. Now replace $(i_h, w(i_h))$ in the
  occurrence of $(p,R)$ by the element of
  $(i'_h, w(i'_h)) \in \widetilde{s}_h \cap G(w)$ having
  the largest $y$-coordinate. (Here we could equally well have chosen the
  largest $x$-coordinate.) Because $s_h$ was shadeable from $h$ this
  will produce a new occurrence of $(p,R)$ in $w$.  Also, note that
  $s_h$ now determines the lines $x=i'_h$ and $y = w(i'_h)$, which are
  to the right of and above, respectively, the previously identified
  lines.

  Thus at each step, we shift at most one horizontal line and at most
  one vertical line, and each such line can only move in a fixed
  direction. The graph $G(w)$ is finite, and so the shifting will have
  to terminate, at which time $\{\widetilde{s}_g : g \in G\} \cap G(w)$
  will be empty. Thus the procedure ends by producing an occurrence of
  $(p,R')$ in $w$.
\end{proof}

The Simultaneous Shading Lemma explains, among other things, the
coincidences discussed in Example~\ref{ex:simultaneous-shading}.
The Simultaneous Shading Lemma together with the Closure Lemma are
powerful enough to explain every coincidence among mesh patterns of
$1$ or $2$ letters, \emph{except} for the coincidence
\[
\gamma_1 \;=\;
\begin{tikzpicture}[scale=.4, baseline=14pt]
  \foreach \x in {(0,1),(0,2),(1,1),(1,2),(2,0)} {\fill[lightgray] \x rectangle ++ (1,1);}
  \foreach \x in {1,2} {
    \draw[gray] (0,\x) -- (3,\x);
    \draw[gray] (\x,0) -- (\x,3);
  }
  \foreach \x in {(1,1),(2,2)} {\fill[black] \x circle (5pt);}
\end{tikzpicture}
\;\,\asymp\,\;
\begin{tikzpicture}[scale=.4, baseline=14pt]
  \foreach \x in {(0,2),(1,0),(1,1),(2,0),(2,1)} {\fill[lightgray] \x rectangle ++ (1,1);}
  \foreach \x in {1,2} {
    \draw[gray] (0,\x) -- (3,\x);
    \draw[gray] (\x,0) -- (\x,3);
  }
  \foreach \x in {(1,1),(2,2)} {\fill[black] \x circle (5pt);}
\end{tikzpicture}
\;=\; \gamma_2,
\]
and its symmetries. It is easy to verify that a permutation $w$ contains
either of $\gamma_1$ and $\gamma_2$ if and only if $w$ can be written as
a direct sum $w = u\oplus v$ of two nonempty permutations $u$ and $v$.
Under certain conditions the coincidence between $\gamma_1$ and
$\gamma_2$ implies coincidences of larger patterns, where $\gamma_1$ and
$\gamma_2$ appear as embedded patterns. An example of this is given in
Figure~\ref{fig:2comp}.

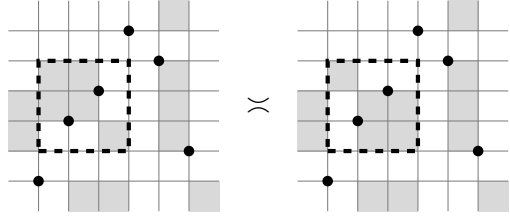
\begin{figure}[htbp]
  \[
  \begin{tikzpicture}[scale=.4, baseline=37pt]
    \foreach \x in {(0,2),(0,3),(1,3),(1,4),(2,0),(2,3),(2,4),(3,0),(3,2),
      (5,2),(5,3),(5,4),(5,6),(6,0)}{
      \fill[lightgray] \x rectangle ++ (1,1);
    }
    \foreach \x in {(1,3),(1,4),(2,3),(2,4),(3,2)}{
      \fill[lightgray] \x rectangle ++ (1,1);
    }
    \foreach \x in {1,2,3,4,5,6} {
      \draw[gray] (0,\x) -- (7,\x);
      \draw[gray] (\x,0) -- (\x,7);
    }
    \draw[ultra thick, dashed] (1,2) rectangle (4,5);
    \foreach \x in {(1,1),(2,3),(3,4),(4,6),(5,5),(6,2)} {\fill[black] \x circle (5pt);}
  \end{tikzpicture}
  \;\;\asymp\;\;
  \begin{tikzpicture}[scale=.4, baseline=37pt]
    \foreach \x in {(0,2),(0,3),(1,4),(2,0),(2,2),(2,3),(3,0),(3,2),(3,3),
      (5,2),(5,3),(5,4),(5,6),(6,0)}{
      \fill[lightgray] \x rectangle ++ (1,1);
    }
    \foreach \x in {(1,4),(2,2),(2,3),(3,2),(3,3)}{
      \fill[lightgray] \x rectangle ++ (1,1);
    }
    \foreach \x in {1,2,3,4,5,6} {
      \draw[gray] (0,\x) -- (7,\x);
      \draw[gray] (\x,0) -- (\x,7);
    }
    \draw[ultra thick, dashed] (1,2) rectangle (4,5);
    \foreach \x in {(1,1),(2,3),(3,4),(4,6),(5,5),(6,2)} {\fill[black] \x circle (5pt);}
  \end{tikzpicture}
  \]
  \caption{Two mesh patterns whose coincidence is due to the coincidence
    of two smaller embedded mesh patterns.}\label{fig:2comp}
\end{figure}

\section{Future research}\label{sec:future}

The tools we derived in this paper go some way toward characterizing
coincidence of mesh patterns, but they are not a complete answer. We
have written tests to apply our results to all mesh patterns of small
lengths, partitioning them into coincident classes, and not all
coincidences can be detected by the methods described in this paper. For example, the following
coincidence cannot be proved by the Simultaneous Shading Lemma, nor by
the additional use of some of the other techniques addressed above.

\begin{ex}
  Consider the following two mesh patterns below.
  $$
  \pi = (123,R) \;=\;
  \begin{tikzpicture}[scale=.4, baseline=21pt]
    \foreach \x in {(0,0),(0,1),(1,0),(2,0),(2,2),(3,0),(3,2),(3,3)} {
      \fill[lightgray] \x rectangle ++(1,1);}
    \foreach \x in {1,2,3} {
      \draw[gray] (\x,0) -- (\x,4);
      \draw[gray] (0,\x) -- (4,\x);
    }
    \foreach \x in {(1,1),(2,2),(3,3)} {\fill[black] \x circle (5pt);}
  \end{tikzpicture}
  \,\;\asymp\;\,
  \begin{tikzpicture}[scale=.4, baseline=21pt]
    \foreach \x in {(0,0),(0,1),(1,0),(2,0),(2,1),(2,2),(3,0),(3,2),(3,3)} {
      \fill[lightgray] \x rectangle ++(1,1);}
    \foreach \x in {1,2,3} {
      \draw[gray] (\x,0) -- (\x,4);
      \draw[gray] (0,\x) -- (4,\x);}
    \foreach \x in {(1,1),(2,2),(3,3)} {\fill[black] \x circle (5pt);}
  \end{tikzpicture}
  \;=\; (123,R') = \pi'
  $$
  Because $R \subset R'$, we have that $\av(\pi) \subseteq \av(\pi')$.
  Now suppose $w$ has an occurrence of $\pi$ indexed by $i_1 < i_2 < i_3$.
  If the region corresponding to the square $(2,1)$ does not
  intersect $G(w)$ for this occurrence, then it is an occurrence of
  $\pi'$ as well. If that region does intersect $G(w)$, then let $g \in
  G(w)$ be the rightmost element of that intersection. If there are no
  elements of $G(w)$ northeast of $g$ and also south of $w(i_2)$ (given
  the choice of $g$, such points would necessarily be in the region
  corresponding to $(3,1)$ for this occurrence), then
  $\{(i_1,w(i_1)),g,(i_3,w(i_3))\}$ is an occurrence of $\pi'$ in
  $w$. On the other hand, suppose there are such points. If they are in
  descending order and $h$ is the leftmost (greatest) among them, then
  $\{(i_1,w(i_1)),g,h\}$ is an occurrence of $\pi'$ in $w$. If they are
  not in descending order, then we can apply the Simultaneous Shading
  Lemma to an ascent among them, producing the occurrence of $\pi'$ that
  we seek.
\end{ex}

We are left, then, with two questions. Under what conditions are two
arbitrary mesh patterns coincident? Is it possible to describe the
maximal set of mesh patterns for which coincidence depends only on
enclosed diagonals?

\end{document}